\documentclass[12pt,a4paper]{article}

\usepackage[latin2]{inputenc}
\usepackage{amssymb}
\usepackage{paralist}
\usepackage{amsmath, calc}
\usepackage{amsfonts}
\usepackage{amsthm}
\usepackage{fullpage}
\usepackage{enumerate}
\usepackage{paralist}

\newtheorem{theorem}{Theorem}[section] %
\newtheorem{lemma}[theorem]{Lemma} %
\newtheorem{problem}[theorem]{Problem} %
\newtheorem{remark}[theorem]{Remark} %

\newtheorem*{theoremA}{Theorem A}
\newtheorem*{theoremB}{Theorem B}
\newtheorem*{theoremC}{Theorem C}
\newtheorem*{theoremD}{Theorem D}
\newtheorem*{theoremE}{Theorem E}
\usepackage{cite}

\begin{document}

\title{On a problem of Nathanson related to minimal asymptotic bases of order $h$}

\author{Shi-Qiang Chen\thanks{School of Mathematics and Statistics,
Anhui Normal University,  Wuhu  241003,  P. R. China;
csq20180327@163.com},
 Csaba
S\'andor \thanks{Department of Stochastics, Institute of Mathematics and Department of Computer Science and Information Theory, Budapest University of
Technology and Economics, M\H{u}egyetem rkp. 3., H-1111 Budapest, Hungary; MTA-BME Lend\"ulet Arithmetic Combinatorics Research Group, ELKH, M\H{u}egyetem rkp. 3., H-1111 Budapest, Hungary. Email: csandor@math.bme.hu. This author was supported by the NKFIH Grants No. K129335.},
Quan-Hui Yang \thanks{School of Mathematics and Statistics, Nanjing University of Information Science and Technology, Nanjing 210044, China; yangquanhui01@163.com.}
}
\date{}
\maketitle

\begin{abstract}
\noindent For integer $h\geq2$ and $A\subseteq\mathbb{N}$, we define
$hA$ to be all integers which can be written as a sum of $h$ elements of $A$.
The set $A$ is called an asymptotic basis of order $h$ if $n\in hA$ for all sufficiently large
integers $n$. An asymptotic basis $A$ of order $h$ is minimal if no proper subset of $A$ is an
asymptotic basis of order $h$. For $W\subseteq\mathbb{N}$, denote by $\mathcal{F}^*(W)$ the set of all finite, nonempty subsets of $W$. Let $A(W)$ be the set of all numbers of the form $\sum_{f \in F} 2^f$, where $F \in \mathcal{F}^*(W)$. In this paper, we give some characterizations of the partitions $\mathbb{N}=W_1\cup\cdots \cup W_h$ with the property that
$A=A(W_1)\cup\cdots \cup A(W_{h})$ is a minimal asymptotic basis of order $h$. This generalizes
a result of Chen and Chen, recent result of Ling and Tang, and also recent result of Sun.


\end{abstract}


\section{Introduction}

Let $\mathbb{N}$ be the set of all nonnegative integers. For integer $h\geq2$ and $A\subseteq\mathbb{N}$, we define
$$hA = \{n: n= a_1+\cdots+a_h, a_i\in A, i = 1, 2,\ldots,h\}.$$
The set $A$ is called an asymptotic basis of order $h$ if $n\in hA$ for all sufficiently large
integers $n$. An asymptotic basis $A$ of order $h$ is minimal if no proper subset of $A$ is an
asymptotic basis of order $h$. This means that for any $a\in A$, the set $E_a = hA\setminus h(A\setminus\{a\})$
is infinite.

Let $W$ be a nonempty subset of $\mathbb{N}$. Denote by $\mathcal{F}^{*}(W)$ the set of all finite, nonempty
subsets of $W$. Let $A(W)$ be the set of all numbers of the form
$\sum\limits_{f\in F}2^f$, where $F\in \mathcal{F}^{*}(W)$.

In 1988, Nathanson \cite{N1988} gave a construction of minimal asymptotic bases of order $h$.
\begin{theoremA}\label{thA}Let $h\geq2$ and let $W_i=\{n\in\mathbb{N}:n\equiv i\pmod {h}\}$ for $i=0, 1,\ldots, h-1$. Let
$A = A(W_0)\cup A(W_1)\cup\cdots\cup A(W_{h-1})$. Then $A$ is a minimal asymptotic
basis of order $h$.
\end{theoremA}
Let $h\geq2$ and $\mathbb{N}=W_1\cup\cdots \cup W_h$ be a partition such that $W_i\cap W_j=\emptyset$ if $i\neq j$. The following all partitions meet this condition.
Nathanson also posed the following open problem:
\begin{problem}\label{pro1}Characterise the partitions $\mathbb{N}=W_1\cup\cdots \cup W_h$ with the property that
$A=A(W_1)\cup\cdots \cup A(W_{h})$ is a minimal asymptotic basis of order $h$?
\end{problem}
In 2011, Chen and Chen \cite{CC2011} solved Problem \ref{pro1} for $h=2$.
\begin{theoremB} Let $\mathbb{N}=W_1\cup W_2$ be a partition with $0\in W_1$ such that $W_1$ and $W_2$ are infinite. Then $A=A(W_1)\cup A(W_2)$ is a minimal asymptotic basis of order $2$ if and only if either $W_1$ contains no consecutive integers or $W_2$ contains consecutive integers or both.
\end{theoremB}
In 2020, Ling and Tang \cite{LT2020} focus on Problem \ref{pro1} for $h=3$, they proved that the following result:
\begin{theoremC} For any $i\in\{0,1,2,3,4,5\}$, if $W_0=\{n\in\mathbb{N}:n\equiv i,i+1\pmod{6}\}$, $W_1=\{n\in\mathbb{N}:n\equiv i+2,i+4\pmod{6}\}$ and $W_2=\{n\in\mathbb{N}:n\equiv i+3,i+5\pmod{6}\}$, then $A=A(W_0)\cup A(W_1) \cup A(W_{2})$ is a minimal asymptotic basis of order three.
\end{theoremC}
In 2021, Sun \cite{S2021} gave a generalization of Theorem A.

\begin{theoremD} Let $h$ and $t$ be two positive integers with $h\geq2$. Let
$$W_j=\bigcup_{i=0}^{\infty}[iht+jt,iht+jt+t-1]$$
for $j=0,1,\cdots,h-1$. Then $A = A(W_0)\cup A(W_1)\cup\cdots\cup A(W_{h-1})$ is a minimal asymptotic basis of order $h$.
\end{theoremD}

In 2011, Chen and Chen gave a sufficient condition for Problem \ref{pro1}.

\begin{theoremE}
Let $h\ge 2$ and $r$ be the least integer with $r>\log h/\log 2$. Let $\mathbb{N}=W_1\cup \dots \cup W_h$ be a partition such that each set $W_i$ is infinite and contains $r$ consecutive integers for $i=1,\dots ,h$. Then $A=A(W_1)\cup \dots \cup A(W_h)$ is a minimal asymptotic basis of order $h$.
\end{theoremE}

For other related results about minimal asymptotic bases, see [2-6,9,11].

In this paper, we continue focus on Problem \ref{pro1}, we give a generalization of Theorem C and Theorem D. For $W\subseteq \mathbb{N}$, set $W(x)=|\{n\in W:n\leq x\}|$.

\begin{theorem}\label{thm1}  Let $h\geq2$ be an integer and $\mathbb{N}=W_1\cup\cdots \cup W_h$ be a partition such that each set $W_j~(1\le j\le h)$ satisfying $|W_j(ht-1)|=t$ for infinitely many integers $t$. Then $A=A(W_1)\cup\cdots \cup A(W_h)$ is a minimal asymptotic basis of order $h$.
\end{theorem}

The second theorem is a generalization of Theorem E.

\begin{theorem}\label{thm2}
Let integer $h\ge 2$ and $\mathbb{N}=W_1\cup \dots \cup W_h$ be a partition such that each set $W_i$ is infinite for $i=1,\dots ,h$, $0\in W_1$ and every $W_i$ contains $\lceil \frac{\log (h+1)}{\log 2}\rceil $ consecutive integers for $i=2,\dots ,h$. Then $A=A(W_1)\cup \dots \cup A(W_h)$ is a minimal asymptotic basis of order $h$.
\end{theorem}
\begin{remark} It is easy to see that $\lceil \frac{\log (h+1)}{\log 2}\rceil $ is the least integer greater than $\log h/\log 2$.
\end{remark}

\section{Lemmas}
\begin{lemma}(See \cite[Lemma 1]{N1988}.)\label{lem1}
Let $\mathbb{N}=W_1\cup\cdots \cup W_h$, where $W_i\neq \emptyset$ for $i=1,\ldots,h$. Then $A=A(W_1)\cup\cdots \cup A(W_h)$ is an asymptotic basis of order $h$.
\end{lemma}

\begin{lemma}\label{lem2} Let $w_1,\ldots,w_s$ be $s$ distinct nonnegative integers. If
$$\sum\limits_{i=1}^{s}2^{w_i}\equiv\sum\limits_{j=1}^{t}2^{x_j}\pmod{2^{w_s+1}},$$
where $0\leq x_1,\ldots, x_t<w_s+1$ are integers that not necessarily distinct, then there exist nonempty disjoint sets $J_1,\ldots,J_s$ of $\{1,2,\ldots,t\}$ such that
$$2^{w_i}=\sum\limits_{j\in J_i}2^{x_j}$$
for $i=1,\ldots,s$.
\end{lemma}
\begin{proof}By the proof of Lemma 2 from \cite{N1988}, there exist nonempty subsets $J_1,\ldots,J_s$ of $\{1,2,\ldots,t\}$ such that
$$2^{w_i}=\sum\limits_{j\in J_i}2^{x_j}$$
for $i=1,\ldots,s$.

The result is trivial for $s=1$. Now we assume that $s\geq 2$. Since there exists a subset $J_1$ of $\{1,2,\ldots,t\}$ such that
$$2^{w_1}=\sum\limits_{j\in J_1}2^{x_j},$$
it follows that
$$\sum\limits_{2\leq i\leq s}2^{w_i}\equiv\sum\limits_{j\in\{1,\ldots,t\}\setminus J_1}2^{x_j}\pmod{2^{w_s+1}},$$
which implies that there exists a subset $J_2$ of $\{1,2,\ldots,t\}\setminus J_1$ such that
$$2^{w_2}=\sum\limits_{j\in J_2}2^{x_j},$$
and so $J_1\cap J_2=\emptyset$. Continuing this process, it follows that there exist nonempty disjoint subsets $J_1,\ldots,J_s$ of $\{1,2,\ldots,t\}$ such that
$$2^{w_i}=\sum\limits_{j\in J_i}2^{x_j},$$
for $i=1,\ldots,s$.

This completes the proof of Lemma \ref{lem2}.
\end{proof}

\section{Proof of Theorem \ref{thm1}}By Lemma \ref{lem1}, it follows that $A$ is an asymptotic basis of order $h$.
For any $a\in A$, there exists $j\in\{1,2,\ldots,h\}$ such that $a\in A(W_j)$. Without loss of generality, we may assume that $a\in A(W_1)$. Then there exists a set $K\subseteq W_1$ such that
$a=\sum\limits_{i\in K}2^{i}$.
Since there exist infinitely many $t$ such that $$|W_j(ht-1)|=t$$ for $j=1,2,\ldots,h$, it follows that there exists an integer $t_{n}$ such that
$$t_{n}h\leq \min K<t_{n+1}h.$$
Let
$$n_T=a+\sum\limits_{j=2}^{h}\sum\limits_{v\in W_j\cap[0,t_{n+1}h-1]}2^v+\sum\limits_{j=2}^{h}\sum\limits_{v\in W_j\cap[t_{n+1}h,T]}2^v,$$
where $T$ is an integer such that $2^T>a$.
Next we shall prove $n_T\in E_a$. Noting that $K(t_{n+1}h-1)\neq\emptyset$ and
$$n_T=\sum\limits_{j\in K(t_{n+1}h-1)}2^{j}+\sum\limits_{j=2}^{h}\sum\limits_{v\in W_j\cap[0,t_{n+1}h-1]}2^v+2^{t_{n+1}h}m$$
for some $m\geq0$.
Let $n_T=b_1+b_2+\cdots+b_h$ be any representation of $n_T$ as a sum of $h$ elements of $A$ and let
$$b_i=\sum\limits_{i\in S_i}2^{i}$$
for $i=1,2\ldots,h$.
Let $$c_i=\sum\limits_{j\in S_i(t_{n+1}h-1)}2^{j}$$ for $i=1,2,\ldots,h$. Then $$c_i\equiv b_i\pmod{2^{t_{n+1}h}}$$ and $$|S_i(t_{n+1}h-1)|\leq t_{n+1}$$ for $i=1,2,\ldots,h$, which implies that
$$\sum\limits_{k\in K(t_{n+1}h-1)}2^{k}+\sum\limits_{j=2}^{h}\sum\limits_{v\in W_j\cap[0,t_{n+1}h-1]}2^v\equiv \sum\limits_{1\leq i\leq h}c_i\equiv\sum\limits_{1\leq i\leq h}\sum\limits_{j\in S_i(t_{n+1}h-1)}2^{j}\pmod{2^{t_{n+1}h}}.$$
Let $$\sum\limits_{k\in K(t_{n+1}h-1)}2^{k}+\sum\limits_{j=2}^{h}\sum\limits_{v\in W_j\cap[0,t_{n+1}h-1]}2^v\equiv\sum\limits_{j=1}^{s}2^{x_j}\pmod{2^{t_{n+1}h}},$$
where $s$ is an integer such that $s\leq t_{n+1}h$. By Lemma \ref{lem2}, there exist nonempty disjoint subsets $J_0,J_1,\ldots,J_{t_{n+1}(h-1)}$ of $\{1,2,\ldots,s\}$ such that
$$2^{\min K}+\sum\limits_{j=2}^{h}\sum\limits_{v\in W_j\cap[0,t_{n+1}h-1]}2^v=\sum\limits_{j\in J_0}2^{x_j}+\sum\limits_{j\in J_1}2^{x_j}+\sum\limits_{j\in J_2}2^{x_j}+\cdots+\sum\limits_{j\in J_{t_{n+1}(h-1)}}2^{x_j}.$$
Therefore,
$$1+t_{n+1}(h-1)\leq1+|J_1|+\cdots+|J_{t_{n+1}(h-1)}|\leq s\leq t_{n+1}h,$$
and so
\begin{equation}\label{equation1}t_{n+1}(h-1)\leq|J_1|+\cdots+|J_{t_{n+1}(h-1)}|\leq t_{n+1}h-1.\end{equation}
Since $|W_j(t_{n+1}h-1)|=t_{n+1}$ for any $j\geq2$, it follows from \eqref{equation1} and Lemma \ref{lem2} that for any $j\geq 2$, there exist $w\in W_j$, $J\subseteq \{1,2,\ldots,s\}$ and $|J|=1$ such that
$$2^w=\sum\limits_{j\in J}2^{x_j},$$
which implies that $x_j=w\in W_j$,
and so $$\{b_1,\ldots,b_h\}\not\subseteq\bigcup\limits_{j=1,j\neq m}^{h}A(W_j)$$ for any $m\geq 2$.
Renumbering the indexes, we always assume that $$b_i\in A(W_i),~~~i\geq2.$$
Then
$$b_1=a+\left(\sum\limits_{j=2}^{h}\sum\limits_{v\in W_j\cap[0,t_{n+1}h-1]}2^v+\sum\limits_{j=2}^{h}\sum\limits_{v\in W_j\cap[t_{n+1}h,T]}2^v-\sum\limits_{2\leq i\leq h}b_i\right).$$
Since the binary representation of $b_1$ is unique, it follows that $b_1=a$, that is $n_T\in E_a$. Noting that $T$ is infinite, we have $A$ is minimal.

This completes the proof of Theorem \ref{thm1}.

\section{Proof of Theorem \ref{thm2}}By Lemma \ref{lem1}, it follows that $A$ is an asymptotic basis of order $h$. For $a\in A$, we assume that $a\in A(W_i)$. Let $E_a=hA\setminus h(A\setminus \{ a\} )$. Now we prove that $E_a$ is infinite.

Let
$$
n_T=a+\sum_{w\in [0,T]\setminus W_i}2^w,
$$
where $T$ is an integer with $T>a$ such that each $[0,T]\cap W_j$ contains $\lceil \frac{\log (h+1)}{\log 2}\rceil $ consecutive integers for $j=2,\ldots,h$. To prove $n_T\in E_a$, it suffices to prove that if $n_T=a_1+a_2+\dots +a_h$ with $a_i\in A$ for $1\le i\le h$, then there exists at least one $a_k=a$.

We distinguish two cases according to whether $i=1$ or $2\le i\le h$.

Case 1. $i=1$. Suppose that there exists an integer $j\ge 2$ such that $$\{ a_1,a_2,\dots ,a_h\} \subseteq \bigcup_{1\le l\le h, l\ne j}A(W_j).$$
Let $\{ b+1,b+2,\dots ,b+\lceil \frac{\log (h+1)}{\log 2}\rceil \}\subseteq [0,T]\cap W_j$. Then by Lemma \ref{lem2} there exists $$\{ a_1',\dots ,a_h'\} \subseteq \bigcup_{1\le l\le h, l\ne j}A(W_j)\cup \{0\}$$ such that
$$
2^{b+1}+\cdots +2^{b+\lceil \frac{\log (h+1)}{\log 2}\rceil}=a_1'+\cdots +a_h'.
$$
Since $a_i'\notin A(W_j)$ for $i=1, \dots ,h$, we have $$a_i'\le 2^0+2^1+\cdots +2^b=2^{b+1}-1$$ for $i=1,\dots ,h$. It follows that
$$
2^{b+1}+\cdots +2^{b+\lceil \frac{\log (h+1)}{\log 2}\rceil}\le h(2^{b+1}-1)<h2^{b+1},
$$
that is $2^{\lceil \frac{\log (h+1)}{\log 2}\rceil}-1<h$, a contradiction. Hence, for any integer $j\ge 2$, we have
$$
\{ a_1,a_2,\ldots ,a_h\} \nsubseteq \bigcup_{1\le l\le h, l\ne j}A(W_l).
$$
Renumbering the indexes, we may assume that $a_i\in A(W_i)$ for $i=2,3,\dots ,h$. It follows that
$$
a_1=a+\sum_{2\le j\le h}(\sum_{w\in [0,T]\cap W_j}2^w-a_j).
$$
Since $a\in A(W_1)$, $W_1,\dots ,W_h$ are disjoint, and the binary representation of $a_1$ is unique, we have $a_1=a$. Therefore $n_T\in E_a$, and $E_a$ is infinite.

Case 2. $i\ge 2$. Since $n_T$ is odd, therefore we may assume that $a_1\in A(W_1)$. Let $2\le j\le h$, $j\ne i$. Similar to the Case 1, we get that
$$
\{ a_1,\dots ,a_h\} \nsubseteq \bigcup_{1\le k\le h, k\ne j}A(W_k).
$$
Renumbering the indexes, we may assume that $a_j\in A(W_j)$ for $j=2,3,\dots ,i-1,i+1,\dots ,h$. It follows that
$$
a_i=a+\sum_{1\le j\le h, j\ne i}(\sum_{w\in [0,T]\cap W_j}2^w-a_j).
$$
Since $a\in A(W_i)$, $W_1,\dots ,W_h$ are disjoint, and the binary representation of $a_i$ is unique, we have $a_i=a$. Therefore $n_T\in E_a$, and $E_a$ is infinite.

This completes the proof.

\end{document}